\documentclass[12pt,a4paper]{amsart} 

\usepackage{latexsym}
\usepackage{geometry}
\usepackage{pinlabel}         
\usepackage{graphicx}
\usepackage{amssymb, amsthm}
\usepackage{epstopdf}
\usepackage{romannum}
\usepackage{tikz}
\usepackage{subfig}
\usepackage[arrow, matrix, curve]{xy} 
\usepackage{hyperref}
\usetikzlibrary{arrows}

\theoremstyle{plain}

\newtheorem*{propositionA}{Proposition A}
\newtheorem*{corollaryB}{Corollary B}
\newtheorem*{propositionC}{Proposition C}
\newtheorem*{corollaryD}{Corollary D}
\newtheorem*{maintheorem}{Main Theorem}

\newtheorem*{corollaryE}{Corollary E}

\newtheorem{theorem}{Theorem}[section]
\newtheorem{lemma}[theorem]{Lemma}
\newtheorem{Itheorem}[theorem]{Iwasawa's Structure Theorem}
\newtheorem{vtheorem}[theorem]{Van Dantzig's Theorem}

\newtheorem{corollary}[theorem]{Corollary}
\newtheorem{example}[theorem]{Example}
\newtheorem{examples}[theorem]{Examples}
\newtheorem{proposition}[theorem]{Proposition}
\theoremstyle{definition}

\newtheorem{remark}[theorem]{Remark}
\newtheorem{definition}[theorem]{Definition}
\newtheorem*{convention}{Convention}

\def\CAT{{\rm{CAT}}$(0)$}

\usepackage{graphicx}
\newcommand{\Z}{\ensuremath{\mathbb{Z}}}

\title[]{Abstract group actions of locally compact groups on CAT(0) spaces}

\author{Philip M\"oller and Olga Varghese}
\date{\today}

\address{Philip M\"oller\\
Department of Mathematics\\
University of M\"unster\\ 
Einsteinstra\ss e 62\\
48149 M\"unster (Germany)}
\email{philip.moeller@uni-muenster.de}

\address{Olga Varghese\\
Department of Mathematics\\
University of M\"unster\\ 
Einsteinstra\ss e 62\\
48149 M\"unster (Germany)}
\email{olga.varghese@uni-muenster.de}

\thanks{The work was funded by the Deutsche Forschungsgemeinschaft (DFG, German Research Foundation) under Germany's Excellence Strategy EXC 2044--390685587, Mathematics M\"unster: Dynamics-Geometry-Structure. The authors were also partially supported by (Polish) Narodowe Centrum
	Nauki, UMO-2018/31/G/ST1/02681}

\begin{document}

\pagenumbering{arabic}
\begin{abstract}
We study abstract group actions of locally compact Hausdorff groups on \CAT\ spaces. Under mild assumptions on the action we show that it is continuous or has a global fixed point. This mirrors results by Dudley and Morris-Nickolas for actions on trees. As a consequence we obtain a geometric proof for the fact that any abstract group homomorphism from a locally compact Hausdorff group into a torsion free \CAT\ group is continuous.
\end{abstract}
\maketitle

\section{Introduction}
This article is located in the area of geometric group theory, a field at the intersection of algebra, geometry and topology. The basic principle of geometric group theory is to investigate algebraic properties of infinite groups using geometric and topological methods. In this project we want to understand the ways in which topological groups can act on spaces of non-positive curvature with the focus on automatic continuity. The main idea of automatic continuity is to establish conditions on topological groups $G$ and $H$ under which an abstract group homomorphism $\varphi\colon G\to H$ is necessarily continuous. {\it We will always assume that topological groups have the Hausdorff property.} There are several results in this direction in the literature, see \cite{CorsonBogopolski}, \cite{ConnerCorson}, \cite{CorsonVarghese}, \cite{Dudley}, \cite{KramerVarghese}, \cite{MorrisNickolas}. Here, the group $G$ will be a locally compact group while $H$ will be the isometry group of a \CAT\ space equipped with the discrete topology. 

One of the powerful theorems in this direction is due to Dudley \cite{Dudley}. He proved that any abstract group homomorphism from a locally compact group into a free group is continuous. By the Nielsen-Schreier-Serre Theorem,  a group is free if and only if it acts freely on a tree \cite[I \S 3.3 Theorem 4]{Serre}. 
For example, the free group $F(\left\{x,y\right\})$ acts freely via left multiplication on the following tree. 

\begin{figure}[h]
\begin{tikzpicture}[scale=0.8, transform shape]
\draw (0,-3)--(0,3);
\draw (-3,0)--(3,0);
\draw (-2,-1)--(-2,1);
\draw (-1,-2)--(1,-2);
\draw (2,-1)--(2,1);
\draw (-1,2)--(1,2);
\draw[fill=black]  (0,0) circle (1pt);
\node at (0.2,0.2) {\tiny{$e$}};
\draw[fill=black]  (2,0) circle (1pt);
\node at (1.8,0.2) {\tiny{$x$}};
\draw[fill=black]  (0,2) circle (1pt);
\node at (0.2,2.2) {\tiny{$y$}};
\draw[fill=black]  (-2,0) circle (1pt);
\node at (-1.7,0.2) {\tiny{$x^{-1}$}};
\draw[fill=black]  (0,-2) circle (1pt);
\node at (0.3,-1.8) {\tiny{$y^{-1}$}};
\draw[fill=black]  (2,1) circle (1pt);
\node at (2.3,1.2) {\tiny{$xy$}};
\draw[fill=black]  (2,-1) circle (1pt);
\node at (2.4,-0.8) {\tiny{$xy^{-1}$}};
\draw[fill=black]  (3,0) circle (1pt);
\node at (3.2,0.2) {\tiny{$xx$}};
\draw[fill=black]  (0,-3) circle (1pt);
\node at (0.55,-2.8) {\tiny{$y^{-1}y^{-1}$}};
\draw[fill=black]  (-1,-2) circle (1pt);
\node at (-1.5,-1.8) {\tiny{$y^{-1}x^{-1}$}};
\draw[fill=black]  (1,-2) circle (1pt);
\node at (1.4,-1.8) {\tiny{$y^{-1}x$}};
\draw[fill=black]  (-3,0) circle (1pt);
\node at (-3.4,0.2) {\tiny{$x^{-1}x$}};
\draw[fill=black]  (-2,1) circle (1pt);
\node at (-1.6,1.2) {\tiny{$x^{-1}y$}};
\draw[fill=black]  (-2,-1) circle (1pt);
\node at (-1.5,-0.8) {\tiny{$x^{-1}y^{-1}$}};
\draw[fill=black]  (0,3) circle (1pt);
\node at (0.2,3.2) {\tiny{$yy$}};
\draw[fill=black]  (-1,2) circle (1pt);
\node at (-1.3,2.2) {\tiny{$yx^{-1}$}};
\draw[fill=black]  (1,2) circle (1pt);
\node at (1.3,2.2) {\tiny{$yx$}};
\draw[dashed] (3,0)--(3.5,0);
\draw[dashed] (3,0.5)--(3,-0.5);
\draw[dashed] (-3,0)--(-3.5,0);
\draw[dashed] (-3,0.5)--(-3,-0.5);
\draw[dashed] (2,1)--(2,1.5);
\draw[dashed] (1.5,1)--(2.5,1);
\draw[dashed] (-2,1)--(-2,1.5);
\draw[dashed] (-1.5,1)--(-2.5,1);
\draw[dashed] (2,-1)--(2,-1.5);
\draw[dashed] (1.5,-1)--(2.5,-1);
\draw[dashed] (-2,-1)--(-2,-1.5);
\draw[dashed] (-1.5,-1)--(-2.5,-1);
\draw[dashed] (0,3)--(0,3.5);
\draw[dashed] (-0.5,3)--(0.5,3);
\draw[dashed] (0,-3)--(0,-3.5);
\draw[dashed] (-0.5,-3)--(0.5,-3);
\draw[dashed] (-1,2)--(-1.5,2);
\draw[dashed] (-1,1.5)--(-1,2.5);
\draw[dashed] (-1,-2)--(-1.5,-2);
\draw[dashed] (-1,-1.5)--(-1,-2.5);
\draw[dashed] (1,2)--(1.5,2);
\draw[dashed] (1,1.5)--(1,2.5);
\draw[dashed] (1,-2)--(1.5,-2);
\draw[dashed] (1,-1.5)--(1,-2.5);
\node at (-7,0) {$F(\left\{x,y\right\})$};
\node at (-5,0) {\Large{$\curvearrowright$}};
\end{tikzpicture}
\end{figure}

Hence, the automatic continuity result by Dudley translates into geometric group theory as follows:  Any abstract action of a locally compact group on a simplicial tree that is via hyperbolic isometries is continuous. 

This result was generalized by Morris and Nickolas in \cite{MorrisNickolas}. They studied abstract group homomorphisms from locally compact groups into free products of arbitrary groups, i.e. replacing each factor $\mathbb{Z}$ of a free group by an arbitrary group. Morris and Nickolas proved that every abstract group homomorphism from a locally compact group into a free product of groups is either continuous, or the image of the homomorphism is conjugate to a subgroup of one of the factors. Let us discuss this result in the special case where the target group is a free product of two groups $A$ and $B$. By the Bass-Serre Theory \cite[I \S 4.1 Theorem 7]{Serre}, there is a simplicial tree $T_{A*B}$ on which $A*B$ acts simplicially without a global fixed point. The Bass-Serre tree $T_{A*B}$ is constructed as follows: the vertices of $T_{A*B}$ are cosets of $A$ and $B$ and two vertices $gA$ and $hB$, $g,h\in A*B$ are connected by an edge iff $gA\cap hB\neq\emptyset$.

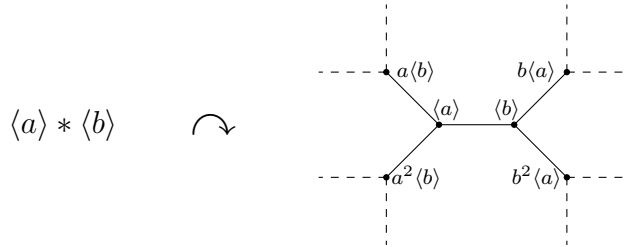
\begin{figure}[h]
\begin{center}
\begin{tikzpicture}
\draw[fill=black]  (0,0) circle (1pt);
\node at (0.1,0.2) {\tiny{$\langle a\rangle$}};
\draw[fill=black]  (1,0) circle (1pt);
\node at (0.9,0.2) {\tiny{$\langle b\rangle$}};
\draw (0,0)--(1,0);
\draw[fill=black]  (-0.7,0.7) circle (1pt);
\node at (-0.3,0.7) {\tiny{$a\langle b\rangle$}};
\draw (0,0)--(-0.7,0.7);
\draw[fill=black]  (-0.7,-0.7) circle (1pt);
\node at (-0.3,-0.7) {\tiny{$a^2\langle b\rangle$}};
\draw (0,0)--(-0.7,-0.7);
\draw[fill=black]  (1.7,0.7) circle (1pt);
\node at (1.3,0.7) {\tiny{$b\langle a\rangle$}};
\draw (1,0)--(1.7,0.7);
\draw[fill=black]  (1.7,-0.7) circle (1pt);
\node at (1.3,-0.7) {\tiny{$b^2\langle a\rangle$}};
\draw (1,0)--(1.7,-0.7);
\draw[dashed] (1.7,0.7)--(1.7,1.6);
\draw[dashed] (1.7,0.7)--(2.6,0.7);
\draw[dashed] (1.7,-0.7)--(1.7,-1.6);
\draw[dashed] (1.7,-0.7)--(2.6,-0.7);
\draw[dashed] (-0.7,0.7)--(-0.7,1.6);
\draw[dashed] (-0.7,0.7)--(-1.6,0.7);
\draw[dashed] (-0.7,-0.7)--(-0.7,-1.6);
\draw[dashed] (-0.7,-0.7)--(-1.6,-0.7);
\node at (-5,0) {$\langle a\rangle * \langle b\rangle$};
\node at (-3,0) {\Large{$\curvearrowright$}};
\end{tikzpicture}
\end{center}
\end{figure}

The picture above shows a small part of the tree $T_{\langle a\rangle *\langle b\rangle}$ in the case where the orders of $a$ and $b$ are equal to $3$. The action of $\langle a\rangle * \langle b\rangle$ on $T_{\langle a\rangle * \langle b\rangle}$ is given by left multiplication. In particular, the vertices have conjugates of $A$ or $B$ as their respective stabilizers. Thus, the fixed point set of every non-trivial elliptic isometry is equal to a single vertex. Taking the point of view of geometric group theory, any abstract group action $L\rightarrow A*B \hookrightarrow{\rm Isom}(T_{A*B})$ is continuous or has a global fixed point.   

Inspired by these results the starting point for our investigation is the study of abstract group actions of locally compact groups on simplicial trees. A natural question is the following: 
\begin{quote}
{\em Under which conditions is an abstract group action of a locally compact group on a simplicial tree continuous or has a global fixed point?}
\end{quote}

We obtain the following result which immediately follows from our Main Theorem.
\begin{propositionA}
Let $\Phi\colon L \to{\rm Isom}(T)$ be an abstract group action of a locally compact group on a simplicial tree $T$ by simplicial isometries without inversion. 

If the fixed point set ${\rm Fix}(\Phi(l))$ of every elliptic isometry $\Phi(l)$ in $\Phi(L)-\left\{{\rm id}_T\right\}$ is a finite tree and the diameter of these fixed point sets is uniformly bounded, then $\Phi$ is continuous or $\Phi$ has a global fixed point.
\end{propositionA}

As an application we obtain a geometric proof of the result by Morris and Nickolas.
\begin{corollaryB}
Any abstract group homomorphism from a locally compact group into a free product is continuous or the image of the homomorphism is conjugate to a subgroup of one of the factors.
\end{corollaryB}

Inspired by the fact that simplicial trees belong to the class of finite dimensional \CAT\ cube complexes, we ask the following more general question: 

\begin{quote}
{\em Under which conditions is an abstract group action of a locally compact group on a finite dimensional \CAT\ cube complex continuous or has a global fixed point?}
\end{quote}

We obtain the following statement for actions on \CAT\ cube complexes which is a special case of our Main Theorem.
\begin{propositionC}
Let $\Phi\colon L\to{\rm Isom}(X)$ be an abstract group action of a locally compact group $L$ on a finite dimensional \CAT\ cube complex $X$ by cubical isometries. 

If $\Phi(l)\neq id_X$  is a hyperbolic isometry for every $l\in L$, then $\Phi$ is continuous.
\end{propositionC}
For example, the free abelian group $\mathbb{Z}^d$ acts freely on $\mathbb{R}^d$ equipped with canonical cubical structure. Thus, any abstract group homomorphism from a locally compact group into $\mathbb{Z}^d$ is continuous.

Let us mention one more application of Proposition C. Right-angled Artin groups are combinatorial generalizations of free and free abelian groups. It is known that a right-angled Artin group $A_\Gamma$ acts freely on the universal cover of the corresponding Salvetti complex $S_\Gamma$ which is a finite dimensional \CAT\ cube complex \cite[Theorem 3.6]{Charney}. Hence, as an immediate application of Proposition C we obtain a geometric proof for the following result, which can be proved by using Dudley's arguments in \cite{Dudley} and can be found in \cite[Cor. 3.13]{CorsonKazachkov}.
\begin{corollaryD}
Any abstract group homomorphism from a locally compact group into a right-angled Artin group is continuous. 
\end{corollaryD}

We want to point out that an action on a finite dimensional \CAT\ cube complex via cubical isometries has the following properties:
\begin{itemize}
\item the action is semi-simple \cite{Bridson}, 
\item the infimum of the translation lengths of hyperbolic isometries is positive \cite{Bridson},
\item any action of a finitely generated group that is via elliptic isometries has a global fixed point \cite{LederVarghese}.
\end{itemize}

This article is dedicated to the more general situation, where we study abstract group actions of locally compact groups on arbitrary \CAT\ spaces with finite flat rank (i.e. there is a bound on the dimension of isometrically embeddable Euclidean space). We prove the following result. 
\begin{maintheorem}
Let $\Phi\colon L\to{\rm Isom}(X)$ be an  abstract group action of a locally compact group $L$ on a complete \CAT\ space $X$ of finite flat rank. 

\begin{enumerate}
\item If $L$ is almost connected (i.e. $L/L^{\circ}$ is compact) and
\begin{enumerate}
\item[(i)] the action is semi-simple,
\item[(ii)] the infimum of the translation lengths of hyperbolic isometries is positive,
\item[(iii)] any finitely generated subgroup of $L$ which acts on $X$ via elliptic isometries has a global fixed point,
\item[(iv)] any subfamily of $\left\{{\rm Fix}(\Phi(l))\mid l\in L\right\}$ with the finite intersection property has a non-empty intersection.
\end{enumerate}
Then $\Phi$ has a global fixed point. 
\item If $L$ is totally disconnected and 
\begin{enumerate} 
\item[(v)] the poset $\left\{{\rm Fix}(\Phi(K))\neq\emptyset\mid K\subseteq L\text{ compact open subgroup}\right\}$ is non-empty and has a maximal element, 
\end{enumerate}
then $\Phi$ is continuous or $\Phi$ preserves a non-empty proper fixed point set ${\rm Fix}(\Phi(K'))$ of a compact open subgroup $K' \subseteq L$.
\item In particular, if $\Phi$ satisfies properties (i)-(iv) and any subfamily of the poset $\left\{{\rm Fix}(\Phi(H))\mid H\subseteq L \text{ closed subgroup}\right\}$ has a maximal element, then $\Phi$ is continuous or $\Phi$ preserves a non-empty proper fixed point set ${\rm Fix}(\Phi(H'))$ of a closed subgroup $H'\subseteq L$. 
\end{enumerate}
\end{maintheorem}

As mentioned before, any action on a tree satisfies properties (i)-(iii). We note that property (iv) is indeed essential, since even the connected locally compact group $\mathbb{R}$ admits an action on a simplicial tree without a global fixed point: There exists an epimorphism $\mathbb{R}\overset{\varphi}{\twoheadrightarrow}\mathbb{Q}$, see Examples \ref{NotContinuous}. Since $\mathbb{Q}$ is denumerable, it is the union of an increasing sequence of finitely generated subgroups $(Q_i)_{i \in I}$. Thus by \cite[Theorem 15]{Serre} the group $\mathbb{Q}$ possesses an action without a global fixed point on a tree $T_{\mathbb{Q}}$ that is constructed out of $(Q_i)_{i \in I}$.  Hence, the action $\mathbb{R}\overset{\varphi}{\twoheadrightarrow}\mathbb{Q}\overset{\psi}{\rightarrow}{\rm Isom}(T_{\mathbb{Q}})$ has no global fixed point. 

Additionally, condition (v) is also necessary as the following example shows. Let $T_3$ be the $3$-regular tree and ${\rm Isom}(T_3)$ its isometry group. The topology of pointwise convergence on vertices gives ${\rm Isom}(T_3)$ the structure of a non-discrete totally disconnected locally compact group. Thus the action $({\rm Isom}(T_3),\mathcal{T}_{pc})\overset{{\rm id}}{\rightarrow}({\rm Isom}(T_3),\mathcal{T}_{discrete})$ is not continuous and there exists no proper invariant fixed point set of a compact open subgroup. But it does not satisfy condition (v): Fixed point sets are convex, so any fixed point set of a compact open subgroup is itself a tree. In the topology of pointwise convergence of vertices the fixed point set of a compact open subgroup has to be bounded. So the fixed point sets of compact open subgroups are precisely the finite subtrees, as every finite subtree is also the fixed point set of a compact open subgroup. Thus there exists no maximal element in the poset of fixed point sets of compact open subgroups.

Before we proceed with an application concerning \CAT\ groups we discuss an example that shows that the condition (v) is not sufficient to get  a {\it global fixed point} or a continuous action. First we construct a homomorphism $\prod\limits_{\mathbb{N}}\Z/2\Z\overset{\Phi}{\twoheadrightarrow}\mathbb{Z}/2\mathbb{Z}$ that is not continuous, see Example \ref{NotContinuous}(iii). Further, the group $\mathbb{Z}/2\mathbb{Z}\times\mathbb{Z}= \langle a, b \mid aba^{-1}b^{-1}, a^2\rangle$ acts on the tree $T$ below as follows: $a$ acts as a reflection on the line, i.e. $a$ maps a red vertex $v$ to a blue vertex $w$ such that the distance between $v$ and $w$ is $2$ and acts on the black vertices as identity and  $b$ acts as a translation on the line. 

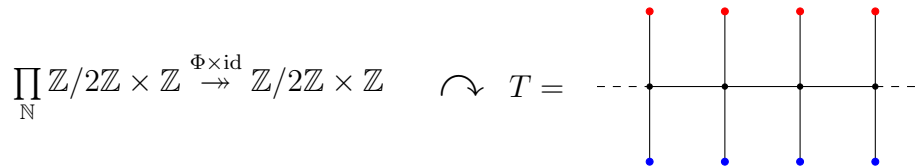
\begin{figure}[h]
\begin{center}
\begin{tikzpicture}
\draw (0,0)--(3,0);
\draw[dashed] (-0.7,0)--(0,0);
\draw[dashed] (3,0)--(3.7,0);
\draw (0,-1)--(0,1);
\draw (1,-1)--(1,1);
\draw (2,-1)--(2,1);
\draw (3,-1)--(3,1);
\draw[fill=black] (0,0) circle (1pt);
\draw[fill=black] (1,0) circle (1pt);
\draw[fill=black] (2,0) circle (1pt);
\draw[fill=black] (3,0) circle (1pt);
\draw[fill=red, draw=red] (0,1) circle (1.3pt);
\draw[fill=red, draw=red] (1,1) circle (1.3pt);
\draw[fill=red, draw=red] (2,1) circle (1.3pt);
\draw[fill=red, draw=red] (3,1) circle (1.3pt);
\draw[fill=blue, draw=blue] (0,-1) circle (1.3pt);
\draw[fill=blue, draw=blue] (1,-1) circle (1.3pt);
\draw[fill=blue, draw=blue] (2,-1) circle (1.3pt);
\draw[fill=blue, draw=blue] (3,-1) circle (1.3pt);
\node at (-6,0) {$\prod\limits_{\mathbb{N}}\Z/2\Z\times\Z\overset{\Phi\times{\rm id}}\twoheadrightarrow\Z/2\Z\times\Z$};
\node at (-2.5,0) {\Large{$\curvearrowright$}};
\node at (-1.5,0) {$T=$};
\end{tikzpicture}
\end{center}
\end{figure}
The group action
$\prod\limits_{\mathbb{N}}\Z/2\Z\times\mathbb{Z}\overset{\Phi\times{\rm id}}\twoheadrightarrow \mathbb{Z}/2\Z\times\mathbb{Z}\hookrightarrow{\rm Isom}(T)$
is not continuous and has no global fixed point, but it preserves a non-empty proper fixed point set of the compact open subgroup $\prod\limits_{\mathbb{N}}\Z/2\Z\times\left\{0\right\}$.

The Main Theorem has fairly broad applications, since a large class of groups is known to act nicely on \CAT\ spaces. 
\begin{corollaryE}
Any abstract group homomorphism $\varphi\colon L\to G$ from a locally compact group $L$ into a \CAT\ group $G$ whose torsion groups are finite is continuous unless the image $\varphi(L)$ is contained in the normalizer of a finite non-trivial subgroup of $G$.

In particular, any abstract group homomorphism from a locally compact group into
\begin{enumerate}
\item[(i)] a right-angled Artin group,
\item [(ii)] a limit group
\end{enumerate}
is continuous.
\end{corollaryE} 

We note that the result of Corollary E was also shown in \cite[Theorem D]{KramerVarghese} using algebraic properties of \CAT\ groups. More precisely, it was shown in \cite{KramerVarghese} that any abstract group homomorphism $\varphi\colon L\to G$, where $G$ does not have an infinite torsion subgroup and every abelian subgroup is a direct sum of cyclic groups, is continuous or the image of $\varphi$ is contained in the normalizer of a finite non-trivial subgroup. Using the fact that an abstact group homomorphism $\varphi:L\rightarrow H$ is continuous if and only if $H$ is torsion-free and does not include $\mathbb{Q}$ or the $p$-adic integers $\mathbb{Z}_p$ for any prime $p$ \cite{CorsonVarghese}, the result in \cite{KramerVarghese} was recently generalized in \cite{Wenzel} to a larger class consisting of all groups $G$ with the following properties: torsion subgroups of $G$ are finite and $\mathbb{Q}$ and the $p$-adic integers $\mathbb{Z}_p$ for any prime $p$ are not a subgroup of $G$. \\

{\bf Outline of the proof of the Main Theorem.} For the first part we show that any action of an abelian group without epimorphisms to $\mathbb{Z}$ on any \CAT\ space with finite flat rank has to be via elliptic isometries. Applying \CAT\ geometry and Iwasawa's Structure Theorem of connected locally compact groups it follows that any action of an almost connected locally compact group on a complete \CAT\ space of finite flat rank has a global fixed point. The second statement follows with an application of a theorem by van Dantzig. We obtain the third result by combining the first and second statement.

We would like to point out that our methods rely heavily on the structure theory for locally compact groups and therefore are not applicable for general completely metrizable groups. We conjecture that a suitable version of the Main Theorem holds for all competely metrizable groups, moreover for all \v{C}ech-complete groups.\\

{\bf Acknowledgment.} The authors thank Linus Kramer for fruitful discussions concerning automatic continuity and non-positive curvature and for helpful comments on a preprint of this paper.

\section{Locally compact groups}
For the study of automatically continuous group homomorphisms we need a group and a topological structure on both sides. To make this combination fruitful, these structures need to harmonize in the following way:
\begin{definition}
  Let $(G,\cdot )$ denote a group and $\mathcal{T}$ a topology on $G$. If the maps $m\colon G\times G\to G$, $m(x,y):=x\cdot y$ and $i\colon G\to G$, $i(x):=x^{-1}$ are continuous with respect to the topology $\mathcal{T}$, then we call $(G,\cdot,\mathcal{T})$ a \textit{topological group}. 
\end{definition}
We will omit the topology from the notation if it does not need to be specified.  
\begin{convention} We assume that topological groups have the Hausdorff property.
\end{convention}
We demand the Hausdorff property as topological groups are Hausdorff if and only if singletons are closed. Topological spaces without that property are in some sense too wild for our discussions.

To obtain nice results about automatic continuity, we need to put some more restrictions on the topologies that our groups carry. A well-studied class of groups is the class consisting of the locally compact groups. Our main reference for these groups is \cite{Stroppel}.
\begin{definition}
  A topological group $(G,\cdot,\mathcal{T})$ is called \textit{(locally) compact}, if the topological space $(G,\mathcal{T})$ is  (locally) compact.
\end{definition}
In general, if $\mathcal{\varepsilon}$ is a topological property, we say that a topological group $G$ \textit{has property} $\mathcal{\varepsilon}$, if the topological space $(G,\mathcal{T})$ has $\mathcal{\varepsilon}$.

Locally compact groups are a natural generalization of real Lie groups and have many interesting examples from different areas of mathematics.
\begin{examples} The following groups are locally compact groups:
  \begin{enumerate}
  \item[(i)] Every group endowed with the discrete topology.
  \item[(ii)] The additive group $\mathbb{R}$ with the standard topology.
  \item[(iii)] The general linear group ${\rm GL}_n(\mathbb{R})$ with the standard topology, more generally: every real Lie group.
  \item[(iv)] The product $\prod\limits_{\mathbb{N}} \Z/2\Z$ with the product topology where each $\Z/2\Z$ carries the discrete topology.
  \item[(v)] The group $\Z_p$ of $p$-adic integers with the $p$-adic topology.
  \end{enumerate}
\end{examples}
In many cases objects in a category are studied with the maps that keep the structure. In this case, these maps are the continuous homomorphisms. We slightly deviate from this process and give the following definition.
\begin{definition}
  Let $G$ and $H$ be topological groups. A not necessarily continuous map $f\colon G\to H$ is called \textit{abstract homomorphism} if it is a group homomorphism from $G$ to $H$.
\end{definition}
Let us discuss some examples of continuous and discontinuous abstract homomorphisms.
\begin{examples}
\label{Homeo}
  The following homomorphisms are continuous:
  \begin{enumerate}
  \item[(i)] The identity map ${\rm id}\colon (\mathbb{R},+,\mathcal{T}_{discrete})\to (\mathbb{R},+,\mathcal{T}_{standard})$.
  \item[(ii)] For any familiy of topological groups $\left(G_i\right)_{i\in I}$, the projection ${\rm pr}_k\colon \prod\limits_{i\in I} G_i\to G_k$ onto the $k$-th coordinate.
  \item[(iii)]   Let $G$ be a topological group. For every $g\in G$, the conjugation map $\gamma_g\colon G\to G$, $x\mapsto gxg^{-1}$.
  \end{enumerate}
\end{examples}

Not all abstract homomorphisms are continuous as the following examples show. Some of these examples are used as counterexamples in the introduction.
\begin{example}
\label{NotContinuous}
 \ \begin{enumerate}
  \item[(i)] The identity map ${\rm id}\colon (\mathbb{R},+,\mathcal{T}_{standard})\to (\mathbb{R},+,\mathcal{T}_{discrete})$ is not continuous even though it is an isomorphism with continuous inverse. 
  \item[(ii)] Consider $\mathbb{R}$ as a $\mathbb{Q}$ vector space with a basis $\{b_i\mid i\in I\}$ with $b_1=1$. We define $\varphi \colon (\mathbb{R},+,\mathcal{T}_{standard}) \twoheadrightarrow (\mathbb{Q},+,\mathcal{T}_{standard})$ as $\varphi(b_j)=\begin{cases}
  1 \quad\textit{if }j=1\\
  0 \quad\textit{else}
  \end{cases}$ and we extend linearly. This obviously defines a surjective homomorphism which can never be continuous since $\mathbb{R}$ is connected and $ \mathbb{Q}$ is totally disconnected.
  \item[(iii)] Consider the compact group $\prod\limits_{\mathbb{N}}\Z/2\Z$. This is a $\Z/2\Z$ vector space and $\bigoplus\limits_{\mathbb{N}} \Z/2\Z$ is a linear subspace. Consider the linear map $\psi\colon \bigoplus\limits_{\mathbb{N}} \Z/2\Z\to \Z/2\Z$ defined by $\psi\left((x_n)_{n\in \mathbb{N}}\right):=\sum_{n\in \mathbb{N}} x_n$. We can extend any basis of $\bigoplus\limits_{\mathbb{N}} \Z/2\Z$ to a basis of $\prod\limits_{\mathbb{N}}\Z/2\Z$ and extend $\psi$ to a linear map $\Phi\colon \prod\limits_{\mathbb{N}}\Z/2\Z\to \Z/2\Z$ by mapping every new basis vector to $0$. This means we have the following diagram:\\
  \begin{center}
  \begin{tikzpicture}
  \coordinate[label=above:$\prod\limits_{\mathbb{N}}\Z/2\Z$] (A) at (0,-0.3);
  \coordinate[label=above:$\Z/2\Z$] (B) at (4,0);
  \coordinate[label=below:$\bigoplus\limits_{\mathbb{N}} \Z/2\Z$] (C) at (0,-2);
 
  \draw[->>] (1,0.4) to node[midway, above]{$\Phi$} (3.3,0.4) ;
  \draw[->] (C) to (0,0.1);
  \draw[->>](0.5,-2) to node[midway, below]{$\psi$}(3.85,0.1);
  \end{tikzpicture}
  \end{center}
  The linearity of the map $\Phi$ guarantees that it is a group homomorphism. However, it can never be continuous, since $\{0\}$ is an open neighbourhood of $0\in \Z/2\Z$, so we would need to find an open set $V$ in $\prod\limits_{\mathbb{N}}\Z/2\Z$, which is mapped to $0$ if $\Phi$ was continuous. But then $V$ needs to contain an open set $U=A_1\times ... \times A_n\times \prod_{N>n}\Z/2\Z$. Then $x=(x_1,\ldots, x_n,y_1,0,0,\ldots)$ and $y=(x_1,\ldots, x_n,y_1+1,0,0,\ldots)$ lie in $U$ and $\Phi(x)=\Phi(y)+1$ contradicting $\Phi(V)\subseteq \{0\}$.
  \item[(iv)] There exists a discontinuous homomorphism $\xi\colon \Z_p\twoheadrightarrow \mathbb{Q}$, where $\Z_p$ is the compact group of $p$-adic integers, see \cite[Proposition 5.8]{Klopsch}. 
  \end{enumerate}
 
\end{example}
Between any two (topological) groups $G$ and $H$ there always exists a (continuous) homomorphism $\varphi\colon G\to H$ by sending every element to the neutral element $e_H$. However there may not exist surjective homomorphisms. As we just saw the group topology sometimes prohibits the existence of surjective continuous homomorphisms. This is however not the only obstacle, as some group structures do not allow surjective homomorphisms themselves. In our case the non-existence of abstract epimorphisms to $\Z$ will be important later on, so let's see some examples of such groups. Recall that a group $G$ is called $k$\textit{-divisible} for a natural number $k>1$ if for every $g\in G$ there exists an $h\in G$ such that $h^k=g$.

\begin{proposition}\label{EpiZ}
  The following groups do not have abstract epimorphisms onto $\Z$:
  \begin{enumerate}
  \item[(i)] $k$-divisible groups for $k>1$,
  \item[(ii)] compact groups,
  \item[(iii)] connected locally compact groups.
  \end{enumerate}
\end{proposition}
\begin{proof}
  Let $G$ be a $k$-divisible group and $\varphi\colon G\to \mathbb{Z}$ a homomorphism. Then $\varphi(G)$ is $k$-divisible too. If $\varphi$ was surjective, then $1\in \varphi(G)$, but $1$ is not $k$-divisible in $\Z$ for any $k>1$, so $\varphi$ cannot be surjective.
 
  For the second (third) part we note that due to Dudley \cite[Theorem 2]{Dudley}, any homomorphism $\psi$ of a compact  (connected locally compact) group $K$ to $\Z$ must be continuous. Thus the image $\psi(K)$ is a finite (connected) subgroup of $\Z$ and hence trivial, contradicting surjectivity.
\end{proof}
There is a different way to prove that compact groups do not have any epimorphisms to $\Z$ via the notion of abelian \textit{algebraically compact} groups. These are abelian groups that can be realized as images of abstract epimorphisms of compact groups. Many properties of these groups are discussed in \cite[Chapter 6]{Fuchs}. If an algebraically compact group is indecomposable (as a direct sum), then it either contains torsion, the $p$-adic integers $\Z_p$ for some prime $p$ or $\mathbb{Q}$ because of \cite[Chapter 6, Corollary 3.6 (ii)]{Fuchs}. The group $\Z$ is indecomposable but does not contain torsion, $\Z_p$ or $\mathbb{Q}$, so it is not algebraically compact, which means there are no abstract epimorphisms from any compact group to $\Z$.

To illustrate the nice interaction between the group and topological structure, we formulate the following lemma, which will be useful later on.
\begin{lemma}(\cite[Lemma 4.4]{Stroppel})
\label{abelian}
  Let $G$ denote a topological group and $A$ an abelian subgroup. Then its closure $\overline{A}$ is an abelian subgroup of $G$.
\end{lemma}
There are two vastly different cases that are considered in the study of locally compact groups: totally disconnected groups and connected groups. The reason for this will become more apparent with the next definition and proposition.

\begin{definition}
Let $G$ be a topological group. The connected component of the identity element (i.e. the union of all connected subsets of $G$ containing $e_G$) is called the \textit{identity component} and will be denoted by $G^\circ$.
\end{definition}

\begin{proposition}(\cite[Lemma 4.9]{Stroppel})
  Let $G$ be a topological group. Then $G^\circ$ is a closed normal subgroup of $G$ and $G/G^\circ$ is a totally disconnected topological group.
\end{proposition}
Hence, for any topological group $G$ the sequence
$\{e_G\}\hookrightarrow G^\circ\overset{\iota}{\hookrightarrow} G \overset{\pi}{\twoheadrightarrow} G/G^\circ \twoheadrightarrow \{e_G\}$ is exact and the maps are continuous.
This gives us the possibility of understanding a general locally compact group by studying a connected and a totally disconnected group, which is sometimes easier than studying the whole group. As an application we obtain a different proof that there is no surjective homomorphism from a compact group onto $\Z$ (compare \cite{Alperin}):

If $\varphi\colon G\to \Z$ is a homomorphism from a compact group onto $\Z$, then it factors through $G/G^\circ$, because $G^\circ$ is compact and therefore divisible, since the structure theory of compact groups  implies the equivalence of connectedness and divisibility \cite[Corollary 2]{Mycielski}. So we have a homomorphism $\varphi'\colon G/G^\circ\to \Z$ of a totally disconnected compact group onto $\Z$, but due to their structure theory its image must be trivial, see \cite[Lemma 5.1]{Bass}.

The different behaviour of connected and totally disconnected groups on an algebraic level will become apparent at the end of the section where the most important theorems are stated. For us the case where the group topology is not only locally compact but in fact compact will be important later on. 

An importan class of locally compact groups is given by the \textit{almost connected} ones.
\begin{definition}
 A topological group $G$ is called \textit{almost connected} if the totally disconnected group $G/G^\circ$ is compact.
\end{definition}
This class of groups obviously contains all connected and all compact groups. Further, the group ${\rm GL}_n({\mathbb{R}})$ is an example of an almost connected locally compact group that is neither connected nor compact. 

To finish this section we now state two structure theorems of locally compact groups.

\begin{Itheorem}(\cite[Theorem 13]{Iwasawa})
\label{Iwasawa}
  Let $L$ be a connected locally compact group. Then we can write $L=H_1\cdot H_2\cdot\ldots\cdot H_n\cdot K$, where each $H_j$ is isomorphic to $\mathbb{R}$ and $K$ is a compact group.
\end{Itheorem}

\begin{vtheorem}(\cite[III \S 4, No. 6]{Bourbaki})
\label{vanDantzig}
  Let $L$ be a totally disconnected locally compact group. Then every neighbourhood of the identity contains an open subgroup. In particular, there exist compact open subgroups.
\end{vtheorem}
These theorems will prove to be incredibly useful as the structure of compact or divisible groups prohibits epimorphisms to $\Z$. This will allow us to deduce fixed point properties for every building block of a connected locally compact group and thus of the whole group. This allows us to define an action of a totally disconnected quotient and therefore the use of van Dantzig's Theorem.

\section{CAT(0) Geometry}
For the study of automatic continuity from the geometric perspective, we need to involve some geometry. This will be in the form of group actions on spaces, in which we can do geometry. For us, these will be the \CAT\ spaces. Our general background reference for \CAT\ spaces is \cite{Bridson-Haefliger}. 

We start by reviewing the definition of geodesic spaces. Let $(X, d)$ be a metric space and $x, y\in X$. A geodesic joining $x$ and $y$ is a map $\alpha_{xy}\colon[0, l]\to X$, such that
$\alpha_{xy}(0)=x, \alpha_{xy}(l)=y$ and $d(\alpha_{xy}(t), \alpha_{xy}(t')) =|t-t'|$ for all $t, t'\in [0, l]$. The image
of $\alpha_{xy}$ is called a {\em geodesic segment}. A metric space $(X, d)$ is
said to be a {\em geodesic space} if every two points in $X$ are joined by a geodesic. Note that geodesics  need not be unique in geodesic metric spaces, but they are unique in $(\mathbb{R}^2, d_2)$ where $d_2$ is the Euclidean metric. 
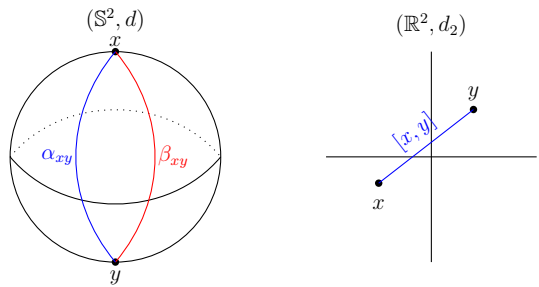
\begin{figure}[h!]
\begin{center}
\begin{tikzpicture}[scale=0.7, transform shape]
\node at (-2,2.6) {$(\mathbb{S}^2, d)$}; 
\node at (4,2.5) {$(\mathbb{R}^2, d_2)$}; 
\draw (-2,0) circle (2cm);
\draw (4,2) -- (4,-2);
\draw (2,0)-- (6,0);
\draw (-4,0) to[bend right=50] (0,0);
\draw[dotted] (-4,0) to[bend left=50] (0,0);
\fill (-2,-2) circle (2pt);
\fill (-2,2) circle (2pt);
\node at (-2,2.2) {$x$};
\node at (-2,-2.3) {$y$};
\draw[color=red] (-2,-2) to[bend right=40] (-2,2);
\draw[color=blue] (-2,-2) to[bend left=40] (-2,2);
\node[color=blue] at (-3.1,0) {$\alpha_{xy}$};
\node[color=red] at (-0.9,0) {$\beta_{xy}$};
\fill (3,-0.5) circle (2pt);
\fill (4.8,0.9) circle (2pt);
\draw[color=blue] (3,-0.5) to node[midway, above, sloped]{$[x,y]$} (4.8,0.9);
\node at (3,-0.9) {$x$};
\node at (4.8,1.2) {$y$};
\end{tikzpicture}
\end{center}
\end{figure}

We say that $X$ is {\em uniquely geodesic} if there is exactly one geodesic joining $x$ and $y$ for
all $x, y\in X$. In that case we denote the image of the geodesic between  $x$ and $y$ by $[x,y]$.

A {\em geodesic triangle} in a geodesic space $X$ consists of three points $p_1, p_2 , p_3$ in $X$ and a choice
of three geodesic segments $\alpha_{p_1p_2}([0,k]), \alpha_{p_2p_3}([0,l]), \alpha_{p_3p_1}([0,m])$. Such a geodesic triangle will
be denoted by $\Delta(p_1, p_2, p_3, \alpha_{p_1p_2}([0,k]), \alpha_{p_2p_3}([0,l]), \alpha_{p_3p_1}([0,m]))$. A triangle $\Delta(\overline{p_1}, \overline{p_2}, \overline{p_3}, [p_1, p_2], [p_2, p_3], [p_1, p_3])$ in Euclidian space $(\mathbb{R}^2, d_2)$ is called
a {\em comparison triangle} for $\Delta(p_1, p_2, p_3, \alpha_{p_1p_2}([0,k]), \alpha_{p_2p_3}([0,l]), \alpha_{p_3p_1}([0,m]))$ if
$d(p_i, p_j)= d_2(\overline{p_i}, \overline{p_j})$ for $i,j=´1, 2, 3$. A point $\overline{x}$ in $[\overline{p_i}, \overline{p_j}]$ is called a {\em comparison point}
for $x\in\alpha_{p_ip_j}([0,n])$ if $d(x, p_i) = d_2(\overline{x}, \overline{p_i})$, which implies $d(x, p_j)= d_2(\overline{x}, \overline{p_j})$. 

A geodesic triangle in $X$
is said to satisfy the {\em CAT(0) inequality} if for all $x$ and $y$ in the geodesic triangle
and all comparison points $\overline{x}$ and $\overline{y}$, the inequality $d(x,y)\leq d_2(\overline{x}, \overline{y})$ holds.

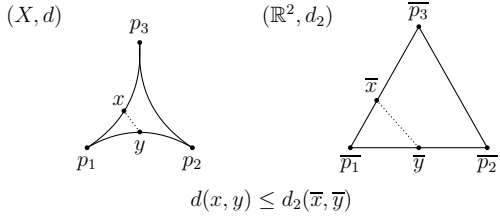
\begin{figure}[h]
\begin{center}
\begin{tikzpicture}[scale=0.7, transform shape]
\node at (-1,2.5) {$(X, d)$}; 
\node at (4,2.5) {$(\mathbb{R}^2, d_2)$}; 
\draw[fill=black]  (0,0) circle (1pt);
\node at (0,-0.3) {$p_1$}; 
\draw[fill=black]  (2,0) circle (1pt);
\node at (2,-0.3) {$p_2$}; 
\draw[fill=black]  (1,2) circle (1pt);
\node at (1,2.3) {$p_3$}; 
\draw (0,0) to [bend left] (2,0);
\draw (0,0) to [bend right] (1,2);
\draw (1,2) to [bend right] (2,0);
\draw[fill=black]  (1,0.3) circle (1pt);
\node at (1,0.0) {$y$}; 
\draw[fill=black]  (0.7,0.7) circle (1pt);
\node at (0.6,0.9) {$x$}; 
\draw[densely dotted] (1,0.3)--(0.7, 0.7);
\draw[fill=black]  (5,0) circle (1pt);
\node at (5,-0.3) {$\overline{p_1}$}; 
\draw[fill=black]  (7.6,0) circle (1pt);
\node at (7.6,-0.3) {$\overline{p_2}$}; 
\draw[fill=black]  (6.3,2.3) circle (1pt);
\node at (6.3,2.6) {$\overline{p_3}$};
\draw (5,0)--(7.6,0);
\draw (5,0)--(6.3, 2.3);
\draw (7.6,0)--(6.3, 2.3);
\draw[fill=black]  (6.3,0) circle (1pt);
\node at (6.3,-0.3) {$\overline{y}$}; 
\draw[fill=black]  (5.5,0.9) circle (1pt);
\node at (5.4,1.2) {$\overline{x}$}; 
\draw[densely dotted] (6.3,0)--(5.5, 0.9);
\node at (3.5,-1) {$d(x,y)\leq d_2(\overline{x},\overline{y})$}; 
\end{tikzpicture}
\end{center}
\end{figure}

\begin{definition}
A metric space $X$ is called a {\em CAT(0) space} if $X$ is a geodesic
space and all of its geodesic triangles satisfy the \CAT\ inequality.
\end{definition}

Roughly speaking, this definition means that triangles in $X$ are not thicker than in the Euclidean plane. One can easily verify from the definition that \CAT\ spaces
are uniquely geodesic. A subset $Y$ of a \CAT\ space $X$ is called
{\em convex} if for all $x$ and $y$ in $Y$ the geodesic segment $[x,y]$ is contained in $Y$. Indeed, a convex subspace of a \CAT\ space is \CAT .
The class of \CAT\ spaces is enormous. Here are some examples.
\begin{examples} The following are \CAT\ spaces:
\begin{enumerate}
\item[(i)] Euclidean and more generally Hilbert spaces.
\item[(ii)] Simplicial trees, metrizing each edge of a simplicial tree as an interval with length one, and more generally metric trees \cite[II Examples 1.15]{Bridson-Haefliger}.
\item[(iii)] Simply connected complete Riemannian manifolds of non-positive sectional curvature \cite[II Theorem 1A.6 and Theorem 4.1]{Bridson-Haefliger}.
\end{enumerate}
\end{examples}

\begin{definition}\label{finDIM}
For a \CAT\ space $X$ we define the \textit{flat rank} of $X$ as the supremum of the dimension of isometrically embedded Euclidean space $\mathbb{R}^n\hookrightarrow X$.
\end{definition}

For example bounded \CAT\ spaces have flat rank $0$, while unbounded simplicial trees have flat rank $1$.

For the next proposition we need to recall the definition of properness of metric spaces. By definition, a metric space $X$ is called {\em proper} if closed balls are compact. In particular, such a space is locally compact and complete. For example, $\mathbb{R}^n$ is a proper metric space, while the infinite dimensional Hilbert space is not proper.

\begin{proposition}(\cite[II Lemma 7.4]{Bridson-Haefliger})\label{FinDim}
Let $X$ be a proper \CAT\ space. If there exists a compact subset $K\subseteq X$ such that $X=\bigcup\limits_{f\in{\rm Isom}(X)} f(K)$, then $X$ has finite flat rank.
\end{proposition}

\begin{convention} We assume from now on that all \CAT\ spaces are complete.
\end{convention}

Our focus lies on isometric group actions of a group $G$ on a \CAT\ space $X$. We describe this actions via the corresponding homomorphism $\varphi\colon G\to{\rm Isom}(X)$ into the isometry group of $X$.
\begin{definition}
Let $(X, d)$ be a \CAT\ space and $f$ be an isometry of $X$. The {\em displacement
function of $f$} is the function $d_f:X\rightarrow\mathbb{R}_{\geq 0}$ defined by $x\mapsto d(f(x), x)$. The {\em translation length of $f$} is defined as
$| f| :={\rm inf}\left\{d_f(x)\mid x\in X\right\}$.
The set of points in $X$ where $d_f$ attains this infimum will be denoted by ${\rm Min}(f)$ and is called the {\em minset of $f$}.
By definition, an isometry is  {\em semi-simple} if ${\rm Min}(f)$ is non-empty.

An isometry $f$ is called
\begin{enumerate}
\item[(i)] {\em elliptic} if $| f| = 0$ and ${\rm Min}(f)\neq \emptyset$.
\item[(ii)] {\em hyperbolic} if $| f| > 0$ and ${\rm Min}(f)\neq\emptyset$.
\item[(iii)] {\em parabolic} if ${\rm Min}(f)=\emptyset$.
\end{enumerate}
\end{definition}

Note that any element in ${\rm Isom}(X)$ falls into exactly one of these classes. An action $\varphi\colon G\to{\rm Isom}(X)$ is called {\em semi-simple} if all its elements in the image of $G$ under $\varphi$ are semi-simple. The minset of a semi-simple isometry carries much useful information about the isometry, see \cite[II Proposition 6.2]{Bridson-Haefliger}. 
\begin{lemma}
\label{invariant}
Let $X$ be a \CAT\ space and $f\in{\rm Isom}(X)$ be a semi-simple isometry.
\begin{enumerate}
\item[(i)] The minset of $f$ is a closed convex set.
\item[(ii)] Let $Y\subseteq X$ be a non-empty complete convex subset. If $Y$ is $f$-invariant, then $| f| =| f |_{Y}$ and ${\rm Min}(f)$ is non-empty iff $Y\cap{\rm Min}(f)$ is non-empty.
\end{enumerate}
\end{lemma}  

The following proposition describes the structure of ${\rm Min}(f)$ for a hyperbolic isometry $f$.
\begin{proposition}(\cite[ II Theorem 6.8]{Bridson-Haefliger})
\label{minset}
Let $X$ be a \CAT\ space and $f\in{\rm Isom}(X)$ be a hyperbolic isometry. 
\begin{enumerate}
\item[(i)] The minset of $f$ is isometric to a product $Y\times\mathbb{R}$, where $Y=Y\times\left\{0\right\}$ is a closed convex subspace of $X$.
\item[(ii)] Every isometry $g\in{\rm Isom}(X)$ that commutes with $f$ leaves ${\rm Min}(f)\cong Y\times\mathbb{R}$ invariant, and $g$ acts on $Y\times\mathbb{R}$ via $(g_1,g_2)$, where $g_1$ is an isometry of $Y$ and $g_2$ is a translation on $\mathbb{R}$.
\end{enumerate}
\end{proposition}

Let $\varphi\colon G\to{\rm Isom}(X)$ be an action on a \CAT\ space $X$.
For us, the \textit{fixed point sets} of subsets of $G$ will play a vital role. These are given by ${\rm Fix}(\varphi(H)):=\left\{x\in X\big| \varphi(h)(x)=x\ \text{for all }h\in H\right\}$, where $H$ is any subset of the group $G$. A quick observation shows that taking fixed point sets is inclusion reversing, which means $H_1\subseteq H_2 \subseteq G$ implies ${\rm Fix}\left(\varphi(G)\right)\subseteq{\rm Fix}\left(\varphi(H_2)\right)\subseteq{\rm Fix}\left(\varphi(H_1)\right)$. Moreover we have:

\begin{lemma}
\label{Fix}
Let $\varphi\colon G\to{\rm Isom}(X)$ be an action of a group G on a \CAT\ space $X$. Let $H\subseteq G$ be a subset, then:
\begin{enumerate}
\item[(i)]  $\bigcap\limits_{h\in H}{\rm Fix}(\varphi(h))={\rm Fix}\left( \varphi(H)\right)={\rm Fix}\left(\langle \varphi(h)\mid h\in H\rangle\right)$.
\item[(ii)] The fixed point set ${\rm Fix}\left(\varphi(H)\right)$ is a closed convex subset of $X$.
\item[(iii)] For any $\alpha\in {\rm Isom}(X)$ we have $\alpha\left({\rm Fix}(\varphi(H))\right)={\rm Fix}\left(\alpha\varphi(H)\alpha^{-1}\right)$.
\item[(iv)] If ${\rm Fix}(\varphi(K_0))$ is a maximal element in the poset $\left\{{\rm Fix}(\varphi(K))\mid K\subseteq G\text{ subgroup}\right\}$, then $\varphi(g){\rm Fix}(K_0)$ is also maximal for all $g\in G$.
\end{enumerate}
\end{lemma}
\begin{proof}
The result in (i) follows straightforward from the definition of the fixed point set. For (ii) we know by Lemma \ref{invariant} that the fixed point set of an elliptic isometry is a closed convex subset. Applying (i) we get  the desired result. The third assertion follows by a straightforward calculation.
\begin{enumerate}
\item[to $(iv)$]
Suppose $\varphi(g){\rm Fix}(\varphi(K_0))\subseteq {\rm Fix}(K_1)$ for some subgroup $K_1\subseteq G$. Applying $\varphi(g)^{-1}$ shows ${\rm Fix}(\varphi(K_0))\subseteq \varphi(g)^{-1}({\rm Fix}(K_1))={\rm Fix}(\varphi(g^{-1}K_1g))$. Now the fact that ${\rm Fix}(\varphi (K_0))$ is maximal together with the fact that $g^{-1}K_1g$ is a subgroup implies that we have "$=$" and not just "$\subseteq$" completing the proof.
\end{enumerate} 
\end{proof}
Note that $(iv)$ remains true if we consider the compact open subgroups of $G$ rather than all subgroups.

For us, the existence of global fixed points, that means a non-empty ${\rm Fix}(\varphi(G))$ will play a big role. The next proposition gives us two conditions that imply this property.

\begin{proposition}(\cite[II Corollary 2.8]{Bridson-Haefliger}, \cite[Lemma 2.1]{Marquis})
\label{boundedOrbit}
Let $\varphi\colon G\to{\rm Isom}(X)$ be a group action on a \CAT\ space $X$.  
\begin{enumerate}
\item[(i)] If for $x\in X$ the subset $\left\{ \varphi(g)(x) \mid g\in G \right\} \subseteq X$ is bounded, then ${\rm Fix}(\varphi(G))$ is non-empty.
\item[(ii)] If $G=H_1\cdot\ldots\cdot H_r$ is a product of subgroups $H_1, \ldots, H_r$ and ${\rm Fix}(\varphi(H_i))\neq\emptyset$ for $i=1,\ldots, r$, then ${\rm Fix}(\varphi(G))$ is non-empty.
\end{enumerate}
\end{proposition}
Now we have gathered enough information about the geometry of \CAT\ spaces and about locally compact groups to start doing geometric group theory and work towards the proof of our Main Theorem.

\section{Compact and Divisible groups}
For the use of Iwasawa's Structure Theorem \ref{Iwasawa} we need more information about actions of compact and abelian divisible groups. We give a geometric proof that these groups have to act via elliptic isometries on \CAT\ spaces of finite flat rank if the action is semi-simple and the infimum of the non-zero translation lengths is positive. Thus, the aim of this section is to prove the following theorem.

\begin{theorem}
\label{abelianNoZ}
Let $\varphi \colon A \to{\rm Isom}(X)$ be a semi-simple action of an abelian group $A$ on a \CAT\ space $X$ with finite flat rank such that $\inf\left\{|\varphi(a)|>0\mid a\in A\right\}>0$. If $A$ has no quotient isomorphic to $\Z$, then every element of $\varphi(A)$ is an elliptic isometry.
\end{theorem}
To illustrate the arguments, we first prove the special case that the \CAT\ space is a simplicial tree in Proposition \ref{Ktrees}. The argument we showcase here is a combination of the arguments given by Morris and Nickolas in \cite{MorrisNickolas} and Alperin in \cite{Alperin}, who uses some results of Bass \cite{Bass}. Before proving this result, we need a small group theoretic result about the structure of the infinite dihedral group which immediately follows from the Kurosh subgroup Theorem \cite{Kurosh}.

\begin{lemma}\label{dihedral}
Let $D_\infty$ denote the infinite dihedral group and let $H\subseteq D_\infty$ be a non-trivial abelian subgroup. Then $H$ is isomorphic to either $\mathbb{Z}/2\mathbb{Z}$ or $\Z$.
\end{lemma}

Now we show the statement of the theorem above in the special case that the \CAT\ space $X$ is a tree. 
\begin{proposition}
\label{Ktrees}
Let $\varphi \colon A \to{\rm Isom}(T)$ be a simplicial action of an abelian group $A$ on a  tree $T$. If $A$ has no quotient isomorphic to $\Z$, then every element of $\varphi(A)$ is an elliptic isometry.  
\end{proposition}
\begin{proof} Without loss of generality, we can assume $T$ is infinite, as else Proposition \ref{boundedOrbit} guarantees the existence of a global fixed point.
	
Suppose for a contradiction, that there exists an element  $a\in A$ such that $\varphi(a)$ is not elliptic. Since $\varphi$ acts on $T$ simplicially, that means that $\varphi(a)$ needs to be a hyperbolic isometry due to \cite[Theorem A]{Bridson}. 
 
For any hyperbolic element, the minset ${\rm Min}(\varphi(a))$ is isometric to a product $Y\times \mathbb{R}$, where $Y$ is a convex subset of $T$, see Proposition \ref{minset}. Since $T$ is an infinite tree, we can deduce immediately that $Y\cong\{y\}$, because $Y$ is convex and an infinite tree has flat rank $1$. Hence ${\rm Min}(\varphi(a))$ is itself a tree (in fact, a doubly infinite line), which we denote by $T(\varphi(a))$.
 
By Proposition \ref{minset} the whole group $A$ acts on $T(\varphi(a))$. Therefore we obtain a homomorphism 
\begin{align*}\psi\colon A&\to{\rm Isom}(T(\varphi(a))),\\
b&\mapsto [x \mapsto \varphi(b)(x)].
\end{align*} 
The isometry group of $T(\varphi(a))$ is itself isomorphic to $D_\infty$. So the image $\psi(A)$ needs to be an abelian  subgroup of $D_\infty$. Lemma \ref{dihedral} tells us that $\psi(A)$ needs to be isomorphic to $\{1\}$, $\Z/2\Z$ or $\Z$. So which one is it?
 
  Since $\varphi(a)$ is a translation on $T(\varphi(a))$, $\psi(A)$ cannot be isomorphic to $\{1\}$ or $\mathbb{Z}_2$. So we deduce $\psi(A)\cong \mathbb{Z}$, contradicting the assumption that $A$ has no quotient isomorphic to $\Z$. This completes the proof.
\end{proof}

Now we prove the general case. Here the minset of a hyperbolic isometry still decomposes as $Y\times \mathbb{R}$ but $Y$ is just a convex subset. This gives us some problems as we need to consider isometries of $Y$ too rather than just translations along $\mathbb{R}$. Some of our arguments are similar to \cite[Theorem 2.5]{Caprace Marquis}, who gave a different proof of a similar result. 

\begin{proof}[Proof of Theorem 4.1] 
Suppose there exists an element $a\in A$ such that $\varphi(a)$ is not elliptic. Then $\varphi(a)$ needs to be a hyperbolic isometry since the action is semi-simple. By Proposition \ref{minset} the minset of $\varphi(a)$ has the structure of a direct product, ${\rm Min}(\varphi(a))\cong Y\times \mathbb{R}$, where $Y$ is a convex subset of $X$ and every element commuting with $\varphi(a)$ leaves ${\rm Min}(\varphi(a))$ invariant. Moreover each element $s\in A$ acts as $\left(s_1,s_2\right)$, where $s_1$ is an isometry of $Y$ and $s_2$ a translation of $\mathbb{R}$. Additionally, the isometry $s_1$ is also semi-simple.
 
We now consider the action of $A$ on $Y$ induced by the above decomposition. Suppose there still exists some $t\in A$ such that $t_1$ is hyperbolic. Then ${\rm Min}(t_1)\cong Y'\times \mathbb {R}$ again and $A$ acts on $Y'\times \mathbb{R}$ since $A$ is abelian. Combining this with the original action we obtain an action of $A$ on the subspace $Y'\times \mathbb{R}\times \mathbb{R}$ via $t\mapsto (t_{1_1},t_{1_2},t_2)$, where the original action of $t$ on $Y\times \mathbb{R}$ was given by $(t_1,t_2)$.
 
Iterating this process as long as there are still elements acting hyperbolically on the first factor gives us an action of $A$ on the subspace $\tilde{Y}\times \mathbb{R}^n$ for some $n\geq1$. For $s\in A$ we write $s= (s_0,s_1,\ldots,s_n)$, where each $s_j$ is a translation of $\mathbb{R}$ for $j\in\{1,...,n\}$ and $s_0$ is some semi-simple isometry of $\tilde{Y}$. This process of splitting up the minset needs to stop eventually, since $X$ has finite flat rank (see Definition \ref{finDIM}). So we can assume that the induced action of $A$ on the subspace $\tilde{Y}$ has no hyperbolic elements, which means every $s_0$ acts elliptically, thus has a fixed point.
 
  The translation length of $|s|$ can be computed using the splitting up into the parts $s_0, s_1,\ldots,s_n$. Since $s_0$ has a fixed point we have $|s_0|=0$. The metric on $\mathbb{R}^n$ is given by the construction in \cite[Proposition II.2.14]{Bridson-Haefliger}. Therefore we obtain
  $$|s|=\sqrt{|s_1|^2+|s_2|^2+...+|s_n|^2}\quad (\star)$$
 
  This gives us an action $\phi$ of $A$ on $\mathbb{R}^n$ via translations and $\inf\left\{|\phi(s)|> 0\mid s\in A\right\}>0$ since this is the case for the original action. We define a homomorphism $\psi\colon A\to (\mathbb{R}^n,+)$ via $\psi (s):= \phi(s)(0)$. Therefore $\psi \left(A\right)=:H$ is a subgroup of $\mathbb{R}^n$. We now differentiate between two cases.
 
 \textit{Case 1: The subgroup $H$ is not closed in $(\mathbb{R}^n, \mathcal{T}_{standard})$}\\
  Since $\mathbb{R}^n$ is locally compact, every discrete subgroup is closed (see for example \cite[Cor. 4.6]{Stroppel}), therefore $H$ is not discrete. That means there needs to be some accumulation point $h\in \mathbb{R}^n$. Since $H$ is a topological group as a subgroup of $\mathbb{R}^n$, we can assume this accumulation point is $0$, because if the group is discrete in $0$, then it is discrete as a group (because leftmultiplication with any element is a homeomorphism). So we have a sequence of elements $\left(h_n\right)_{n\in \mathbb{N}}$, $h_i\neq 0$ for all $i\in \mathbb{N}$ and $h_n$ converges to $0$. Due to the construction of $\psi$ and $(\star)$, we can see that for all $s\in A$ the equation $|\varphi(s)|=|\psi(s)|$ holds. So the existence of the sequence $(h_n)$ contradicts the statement $\inf\left\{|\varphi(s)|>0\mid s\in A\right\}>0$, finishing case 1.
  
  \textit{Case 2: The subgroup $H$ is closed in $(\mathbb{R}^n, \mathcal{T}_{standard})$}\\
  If $H$ is closed, then $H$ is isomorphic to $\mathbb{R}^l\times \Z^m$ for some $l,m\in \mathbb{N}$, see\cite[Ch. 7, Section 1.2]{Bourbaki}. As in the case of an action on a tree, we can deduce that $m=0$ since $A$ has no epimorphisms to $\Z$. If $|\psi(s)|>0$, then we can once again find a series $\left(h_n\right)_{n\in\mathbb{N}}$ converging to $0$ while each $h_i\neq 0$ for $i\in \mathbb{N}$. This once again contradicts $\inf\left\{|\varphi(s)|> 0\mid s\in A\right\}>0$. So $\psi\left(A\right)=\{0\}$. But then we immediately obtain $|\varphi(a)|=0$ contradicting the hyperbolicity of $\varphi(a)$ finishing case 2.
 
Thus we have arrived at a contradiction in each case, meaning our hypothesis that a hyperbolic isometry  exists in $\varphi(A)$ was false, therefore finishing the proof.
\end{proof}
This has a nice corollary for compact groups:
\begin{corollary}\label{cpt}
  Let $\varphi\colon K\to {\rm Isom}(X)$ be a semi-simple action of a compact group $K$ on a \CAT\ space $X$ of finite flat rank. If $\inf\{|\varphi(k)|>0\mid k\in K\}>0$, then every element in $\varphi(K)$ is an elliptic isometry.
\end{corollary}
\begin{proof}
  Assume there exists a hyperbolic element $\varphi(k)\in \varphi(K)$. By Lemma \ref{abelian} the subgroup $\overline{\langle k\rangle}$ is abelian. Hence $\varphi$ induces an action of a compact abelian subgroup $\overline{\langle k\rangle}$ on $X$ with all the assumptions of the previous theorem. Proposition \ref{EpiZ} shows that compact groups admit no epimorphisms to $\Z$, thus every element of $\overline{\langle k\rangle}$ needs to act elliptically, contradicting the assumption that $\varphi(k)$ is hyperbolic.
\end{proof}

\begin{corollary}
\label{R}
  Let $\varphi\colon\mathbb{R}\to \mathrm{Isom}(X)$ be a semi-simple action on a \CAT\ space $X$ of finite flat rank. If $\inf\{|\varphi(r)|>0\mid r\in \mathbb{R}\}>0$, then every element in $\varphi(\mathbb{R})$ is an elliptic isometry.
\end{corollary}
\begin{proof}
Since $\mathbb{R}$ is divisible, Proposition \ref{EpiZ} tells us that it admits no epimorphisms to $\Z$, so the conditions of the previous theorem are satisfied. The conclusion of the theorem then finishes the proof.
\end{proof}

The following example shows the necessity of the condition on the infimum of translation lengths. The group $(\mathbb{R},+)$ acts on $\mathbb{R}$ via $\varphi\colon \mathbb{R}\to {\rm Isom}(\mathbb{R})$, $\varphi(t)(r):=t+r$. Clearly, no element besides $id_{\mathbb{R}}$ is elliptic, but also $\inf\{|\varphi(t)|>0\mid t\in \mathbb{R}\}=0$.

Now that we have concluded the desired local fixed point properties of all components of connected locally compact groups, we can prove the Main Theorem.

\section{The proof of the Main Theorem}
Let us start by proving the first statement of the Main Theorem. Recall that a family $\mathcal{F}$ of sets has the \textit{finite intersection property}, if for every finite subfamily $\emptyset\neq \mathcal{F}_0\subseteq \mathcal{F}$ we have $\bigcap \mathcal{F}_0\neq \emptyset$. 

\begin{proof}[Proof of the Main Theorem (1)]
By Iwasawa's Structure Theorem \ref{Iwasawa}, the identity component $L^{\circ}\subseteq L$ can be decomposed as a product 
\[
L^{\circ}=H_1\cdot\ldots\cdot H_n\cdot K
\]
where each $H_j$ is isomorphic to $\mathbb{R}$ and $K$ is a compact group. We first show that the action 
\[
\Phi_{|H_j}\colon H_j\to{\rm Isom}(X)
\]
has a global fixed point for $j=1,\ldots, n$. By assumptions (i) and (ii) the action $\Phi_{|H_j}$ is semi-simple and ${\rm inf}\left\{|\Phi(h)|>0\mid h\in H_j\right\}>0$, it thus follows by Corollary \ref{R} that the group $H_j$ acts via elliptic isometries on $X$. 

Let us now consider the subfamily $\mathcal{H}_j:=\left\{{\rm Fix}(\Phi(h))\mid h\in H_j\right\}$ of $\left\{{\rm Fix}(\Phi(l))\mid l\in L\right\}$. This family has the finite intersection property because of the following reason: for a finite subset $\mathcal{F}:=\left\{{\rm Fix}(\Phi(h_1)), \ldots, {\rm Fix}(\Phi(h_n))\right\}$ of  $\mathcal{H}_j$ the intersection of $\mathcal{F}$ is equal to ${\rm Fix}(\langle \Phi(h_1),\ldots, \Phi(h_n)\rangle)$ due to Lemma \ref{Fix}. By assumption (iii) this fixed point set is non-empty. Hence, the intersection of $\mathcal{H}_j$ is non-empty by assumption (iv). We have
\[
\emptyset\neq\bigcap\mathcal{H}_j={\rm Fix}(\Phi(H_j))\text{ for }j=1,\ldots, n.
\] 
We can proceed analogously to show that $\Phi_{|K}\colon K\to{\rm Isom}(X)$ has a global fixed point using Corollary \ref{cpt} this time. 

Hence, the action 
\[
\Phi_{|L^{\circ}}:L^{\circ}=H_1\cdot\ldots\cdot H_n\cdot K\rightarrow{\rm Isom}(X)
\]
has a global fixed point in $X$ by Proposition \ref{boundedOrbit}.

We have ${\rm Fix}(\Phi(L^{\circ}))\neq\emptyset$ and by Lemma \ref{Fix} the fixed point set ${\rm Fix}(\Phi(L^{\circ}))$ is a closed convex subspace of $X$, thus a \CAT\ space of finite flat rank. The quotient $L/L^{\circ}$ acts on ${\rm Fix}(\Phi(L^{\circ}))$ via $\varphi(gL^{\circ})(x):=\Phi(g)(x)$. Thus, the induced action 
\[
\varphi:L/L^{\circ}\rightarrow{\rm Isom}({\rm Fix}(\Phi(L^{\circ})))
\]
has properties (i)-(iv) due to Lemma \ref{invariant}. Since $L$ is almost connected, the quotient $L/L^{\circ}$ is compact. By the same argument as in the first part of the proof, the compact group $L/L^{\circ}$ has a global fixed point in ${\rm Fix}(\Phi(L^{\circ}))$. Hence, the action $\Phi$ has a global fixed point in $X$.
\end{proof}

We now complete the proof of the Main Theorem.

\begin{proof}[Proof of the Main Theorem (2)]

Since $L$ is a totally disconnected locally compact group, by van Dantzig's Theorem \ref{vanDantzig} there exists a compact open subgroup $K\subseteq L$. 

If there exists a compact open subgroup $K$ such that $\Phi(K)=\{{\rm id}_X\}$, then the kernel of $\Phi$ is open in $L$. Hence the map $\Phi$ is continuous. 

If there is no compact open subgroup $K\subseteq L$ such that $\Phi(K)$ is trivial, then by assumption the family 
\[
\left\{{\rm Fix}(\Phi(K))\neq\emptyset\mid K\subseteq L\text{ compact open subgroup}\right\}
\]
is non-empty and has a maximal element ${\rm Fix}(\Phi(K_0))\neq X$. We claim that ${\rm Fix}(\Phi(K_0))$ is invariant under $\Phi$. For $g\in L$ the group $gK_0 g^{-1}\cap K_0$ is a compact open subgroup of $L$ by Example \ref{Homeo}(iii).  We have the following relations between the fixed point sets of the compact open subgroups $K_0$, $gK_0 g^{-1}$ and $gK_0 g^{-1}\cap K_0$:
\[
{\rm Fix}(\Phi(K_0))\subseteq{\rm Fix}(\Phi(gK_0 g^{-1}\cap K_0))
\text{ and }
{\rm Fix}(\Phi(gK_0g^{-1}))\subseteq{\rm Fix}(\Phi(gK_0 g^{-1}\cap K_0)).
\]
Since ${\rm Fix}(\Phi(K_0))$ is maximal, the fixed point set $\Phi(g)({\rm Fix}(\Phi(K_0)))={\rm Fix}(\Phi(gK_0g^{-1}))$  is also maximal by Lemma \ref{Fix} and we conclude that
\[
{\rm Fix}(\Phi(K_0))={\rm Fix}(\Phi(gK_0 g^{-1}\cap K_0))
\text{ and }
{\rm Fix}(\Phi(gK_0g^{-1}))={\rm Fix}(\Phi(gK_0 g^{-1}\cap K_0)).
\]
Hence
\[
\Phi(g){\rm Fix}(\Phi(K_0))={\rm Fix}(\Phi(gK_0g^{-1}))={\rm Fix}(\Phi(K_0)).
\]
\end{proof}

\begin{proof}[Proof of the Main Theorem (3)]
By the Main Theorem (1) the action $\Phi_{|L^{\circ}}\colon L^{\circ}\to{\rm Isom}(X)$ has a global fixed point.

Suppose first that ${\rm Fix}(\Phi(L^{\circ}))=X$. Then the action $\Phi$ factors through $L/L^{\circ}$: 
\begin{center}
  \begin{tikzpicture}
  \coordinate[label=above:$L$] (A) at (0,0);
  \coordinate[label=above:${\rm Isom}(X)$] (B) at (4.5,0);
  \coordinate[label=below:$L/L^{\circ}$] (C) at (0,-2);
 
  \draw[->] (0.5,0.4) to node[midway, above]{$\Phi$} (3.5,0.4) ;
  \draw[->>] (A) to (C);
  
\node at (-0.25,-1) {$\pi$};
  \draw[->](0.5,-2) to node[midway, below]{$\varphi$}(3.85,0.1);
  \end{tikzpicture}
  \end{center}
Since the quotient $L/L^{\circ}$ is a totally disconnected locally compact group, by the Main Theorem (2) it follows that $\varphi$ is continuous or $\varphi$ preserves a non-empty proper fixed point set ${\rm Fix}(\varphi(K'))$ where $K'$ is a compact open subgroup of $L/L^{\circ}$. Hence $\Phi$ is continuous or preserves the non-empty proper fixed point set ${\rm Fix}\left(\Phi\left(\pi^{-1}(K')\right)\right)$ of the closed subgroup $\pi^{-1}(K')$.

Suppose now that ${\rm Fix}(\Phi(L^{\circ}))\neq X$. Since  $L^{\circ}$ is a closed normal subgroup of $L$, the non-empty proper fixed point set ${\rm Fix}(\Phi(L^{\circ}))$ is invariant under $\Phi$.
\end{proof}

\begin{remark}
	It is conjectured that property (iii) of the Main Theorem is always satisfied for simplicial actions on finite dimensional \CAT\ simplicial complexes. In the case of some $2$-dimensional \CAT\ triangle complexes this is proven in \cite[Theorem 1.1]{OsajdaQ} and for affine buildings of types $\tilde{A_2}$ and $\tilde{C_2}$ in \cite[Theorem A]{SchillewartStruyveThomas}.
\end{remark}
Before we discuss an application to \CAT\ groups, let's see one example. Let $T_3$ be the $3$-regular tree. At every vertex we attach two additional edges, see picture below.  The group $\mathbb{Z}/2\mathbb{Z}$ acts on this new tree $T'_3$ by flipping the newly attached edges. So we have an action of $\Z/2\Z\times{\rm Isom}(T_3)$  on $T'_3$ where ${\rm Isom}(T_3)$ preserves the colors of the vertices. Combining this with the discontinuous homomorphism $\Phi\colon\prod\limits_{\mathbb{N}}\Z/2\Z \twoheadrightarrow\Z/2\Z$ from Example \ref{NotContinuous}, we get an action of $\prod\limits_{\mathbb{N}}\Z/2\Z\times {\rm Isom}(T_3)$ on $T'_3$: 
\begin{figure}[h]
\begin{center}
\begin{tikzpicture}
\draw[fill=black]  (0,0) circle (1pt);
\draw[fill=black]  (1,0) circle (1pt);
\draw (0,0)--(1,0);
\draw[fill=black]  (-0.7,0.7) circle (1pt);
\draw (0,0)--(-0.7,0.7);
\draw[fill=black]  (-0.7,-0.7) circle (1pt);
\draw (0,0)--(-0.7,-0.7);
\draw[fill=black]  (1.7,0.7) circle (1pt);
\draw (1,0)--(1.7,0.7);
\draw[fill=black]  (1.7,-0.7) circle (1pt);
\draw (1,0)--(1.7,-0.7);
\draw[dashed] (1.7,0.7)--(1.7,1.6);
\draw[dashed] (1.7,0.7)--(2.6,0.7);
\draw[dashed] (1.7,-0.7)--(1.7,-1.6);
\draw[dashed] (1.7,-0.7)--(2.6,-0.7);
\draw[dashed] (-0.7,0.7)--(-0.7,1.6);
\draw[dashed] (-0.7,0.7)--(-1.6,0.7);
\draw[dashed] (-0.7,-0.7)--(-0.7,-1.6);
\draw[dashed] (-0.7,-0.7)--(-1.6,-0.7);
\draw (0,0)--(0,0.7);
\draw (0,0)--(0,-0.7);
\draw[fill=red, draw=red]  (0,0.7) circle (1.3pt);
\draw[fill=blue, draw=blue]  (0,-0.7) circle (1.3pt);
\draw (1,0)--(1,0.7);
\draw (1,0)--(1,-0.7);
\draw[fill=blue, draw=blue]  (1,0.7) circle (1.3pt);
\draw[fill=red, draw=red]  (1,-0.7) circle (1.3pt);
\draw (-1.2,0.1)--(-0.7, 0.7);
\draw (-0.2,1.3)--(-0.7, 0.7);
\draw[fill=red, draw=red]  (-0.2,1.3) circle (1.3pt);
\draw[fill=blue, draw=blue]  (-1.2,0.1) circle (1.3pt);
\draw (2.2,-0.1)--(1.7, -0.7);
\draw (1.2,-1.3)--(1.7, -0.7);
\draw[fill=red, draw=red]  (2.2,-0.1) circle (1.3pt);
\draw[fill=blue, draw=blue]  (1.2,-1.3) circle (1.3pt);
\draw (-0.2,-1.3)--(-0.7, -0.7);
\draw (-1.2,-0.1)--(-0.7, -0.7);
\draw[fill=red, draw=red]  (-0.2,-1.3) circle (1.3pt);
\draw[fill=blue, draw=blue]  (-1.2,-0.1) circle (1.3pt);
\draw (2.2,0.1)--(1.7, 0.7);
\draw (1.2,1.3)--(1.7, 0.7);
\draw[fill=red, draw=red]  (2.2,0.1) circle (1.3pt);
\draw[fill=blue, draw=blue]  (1.2,1.3) circle (1.3pt);
\node at (-8,0) {$\prod\limits_{\mathbb{N}}\Z/2\Z\times{\rm Isom}(T_3)\overset{\Phi\times{\rm id}}\twoheadrightarrow\Z/2\mathbb{Z}\times{\rm Isom}(T_3)$};
\node at (-3.5,0) {\Large{$\curvearrowright$}};
\node at (-2.5,0) {$T'_3:=$};
\end{tikzpicture}
\end{center}
\end{figure}

We give $\prod\limits_{\mathbb{N}}\Z/2\Z$ the product topology with the discrete topology on each factor and ${\rm Isom}(T_3)$ the discrete topology. This makes the group $\prod\limits_{\mathbb{N}}\Z/2\Z\times{\rm Isom}(T_3)$ totally disconnected and locally compact. Condition (v) of the Main Theorem is satisfied, as $\prod_{\mathbb{N}}\Z/2\Z\times\left\{id_{T_3}\right\}$ is a compact open subgroup and it fixes the original tree $T_3$ making it maximal as either all additionally attached edges are flipped or none are. So since the action is not continuous, there needs to be a proper invariant fixed point set. This is given by the original tree $T_3$.

\section{Application}

An important application of the Main Theorem is to the class of \textit{CAT(0) groups} with an additional property controlling the size of torsion subgroups. Similar to \cite{Bridson-Haefliger} we define the following properties of an action on a \CAT\ space.
\begin{definition} Let $\varphi\colon G\to{\rm Isom}(X)$ be a group action on a  \CAT\ space $X$. We say
\begin{enumerate}
\item[(i)] The action is \textit{proper} if for every $x\in X$, there exists an $r\in \mathbb{R}_{>0}$ such that the set $\left\{g\in G\mid \varphi(g)(B_r(x)) \cap B_r(x)\neq \emptyset \right\}$ is finite.
\item[(ii)] The action is \textit{cocompact} if there exists a compact subspace $K\subseteq X$ such that $\varphi(G)(K)=X$.
\end{enumerate}
We call $G$ a \textit{CAT$(0)$ group} if there exists a proper CAT$(0)$ space $Y$ and a proper, cocompact action $\psi\colon G\to {\rm Isom}(Y)$. 
\end{definition} 
In particular, such an action is via semi-simple isometries and the infimum of translation lengths of hyperbolic isometries is positive, see \cite[Prop. II.6.10]{Bridson-Haefliger}.

Some authors require the group action to be free (or faithful) for a \CAT\ group and some others do not assume that the \CAT\ space is proper. For our purposes this definition is suitable as the structure of non-proper \CAT\ spaces is too wild and as we do want to discuss groups that contain torsion. Note that if the action is not faithful, the kernel of the action needs to be finite due to properness.

We now revisit Corollary E from the introduction.
\begin{corollaryE}\label{CAT(0) gr} Any abstract group homomorphism $\varphi\colon L\to G$ from a locally compact group $L$ into a \CAT\ group $G$ whose torsion groups are finite is continuous unless the image $\varphi(L)$ is contained in the normalizer of a finite non-trivial subgroup of $G$.

In particular, any abstract group homomorphism from a locally compact group into
\begin{enumerate}
\item[(i)] a right-angled Artin group,
\item [(ii)] a limit group
\end{enumerate}
is continuous.
\end{corollaryE}
Before giving the proof, we will discuss two lemmas that are helpful for checking the conditions of the Main Theorem.

\begin{lemma}\label{intersection property} Let $\psi\colon G\to{\rm Isom}(X)$ be a faithful proper action on a \CAT\ space $X$. Assume that all torsion subgroups of $G$ are finite. Then every subgroup $H\subseteq G$ acting via elliptic isometries on $X$ is finite and therefore has a global fixed point. 

In particular, conditions (iii) and (iv) of the Main Theorem are satisfied.
\end{lemma}
\begin{proof}
Let $H\subseteq G$ denote a subgroup acting via elliptic isometries. Then each $\psi (h)$ is elliptic and thus has finite order, since the action is proper. So the group $\psi(H)$ is a torsion group and is by assumption finite. Due to Proposition \ref{boundedOrbit}, this means that $\psi(H)$ has a global fixed point. This immediately shows condition (iii) of the Main Theorem.

For condition (iv) suppose the family $\left\{ {\rm Fix}(\psi(h_i))\mid i\in I\right\}$ has the finite intersection property for some index set $I$. If we can show $F:=\langle \psi(h_i)\mid i\in I\rangle $ is finite, then we are done. To do so it is enough to show that every element acts elliptically due to the torsion subgroups being finite. So let $w\in F$. The element $w$ is a finite word, which means $w=\psi\left(h_{i_1}\right)\psi\left(h_{i_2}\right)...\psi\left(h_{i_n}\right)$ for some $i_1,i_2,...,i_n\in I$. But then $w$ fixes $\bigcap\limits_{j=1}^{n}{\rm Fix}\left(\psi\left(h_{i_j}\right)\right)$ which is non-empty, because the family $\left\{ {\rm Fix}(\psi(h_i))\mid i\in I\right\}$ has the finite intersection property, so condition (iv) is satisfied.
\end{proof}

\begin{lemma}\label{max el}
Let $\psi\colon G\to{\rm Isom}(X)$ be a faithful proper cocompact action on a \CAT\ space $X$. Assume that all torsion subgroups of $G$ are finite. For any family of subgroups $\mathcal{H}=\left(H_i\right)_{i\in I}$, the family $\left\{{\rm Fix}(\psi(H_i))\mid i\in I\right\}$ has a maximal element.
\end{lemma} 
\begin{proof}
Let $H_i\in \mathcal{H}$. Using the same argument as in the previous lemma, we can see that if ${\rm Fix}(\psi(H_i))\neq\emptyset$, then  $\psi(H_i)$ and therefore $H_i$ is finite. If the family $\{{\rm Fix}(H_i)\mid i\in I\}$ did not contain a maximal element, then we could find a proper infinite chain 
$${\rm Fix}(\psi(H_{i_1}))\subsetneq {\rm Fix}(\psi(H_{i_2}))\subsetneq  {\rm Fix}(\psi(H_{i_3})) \subsetneq\ldots$$
In particular, we have infinitely many finite subgroups $H_{i_1},H_{i_2},\ldots$ . Since $G$ is a \CAT\ group, it has only finitely many conjugacy classes of finite subgroups (see \cite[III.$\Gamma$.Theorem 1.1]{Bridson-Haefliger}). Therefore at some point we have indices $j\neq k \in I$ and a $g\in G$ such that $H_{i_j}=gH_{i_k}g^{-1}$. Then Lemma \ref{Fix} implies that 
$${\rm Fix}(\psi(H_{i_j}))={\rm Fix}(\psi(gH_{i_k}g^{-1}))=\psi(g) ({\rm Fix}(\psi(H_{i_k})).$$
Since $j\neq k$, either ${\rm Fix}(\psi(H_{i_j}))\subsetneq {\rm Fix}(\psi(H_{i_k}))$ or ${\rm Fix}(\psi(H_{i_k}))\subsetneq {\rm Fix}(\psi(H_{i_j}))$, which is both impossible since each $\psi(g)$ is an isometry.

\end{proof}
\begin{proof}[Proof of Corollary E] Let $X$ be a proper CAT$(0)$ space and $\psi\colon G\to {\rm Isom}(X)$ a proper cocompact group action. We thus obtain an action $\xi\colon L\to {\rm Isom}(X)$ via composition $\xi:=\psi\circ \varphi$. We need to show that (3) of the Main Theorem is applicable.

Without loss of generality we can assume that $\psi$ is injective, because else $\ker(\psi)$ is a finite normal subgroup of $G$ due to properness of the action. Therefore $G={\rm Nor}_G(\ker(\psi))$ and hence the image of $\varphi $ is contained in the normalizer of a finite non-trivial subgroup, namely $\ker (\psi)$.

So from now on we assume that $\psi$ is injective. We have the following situation:
\begin{center}
 \begin{tikzpicture}
  \coordinate[label=above:$L$] (A) at (0,0);
  \coordinate[label=above:${\rm Isom}(X)$] (B) at (4.5,0);
  \coordinate[label=below:$G$] (C) at (0,-2);
 
  \draw[->] (0.5,0.4) to node[midway, above]{$\xi$} (3.5,0.4) ;
  \draw[->] (A) to (C);
  
\node at (-0.25,-1) {$\varphi$};
  \draw[right hook->](0.5,-2) to node[midway, below]{$\psi$}(3.85,0.1);
  \end{tikzpicture}
  \end{center}
Note that $G$ and ${\rm Isom}(X)$ carry the discrete topology and $\psi$ is injective, so $\psi$ is a homeomorphism onto its image.

Since $X$ is proper it is locally compact, thus Proposition \ref{FinDim} implies that it has finite flat rank.
Additionally, the action $\xi$ is semi-simple and $\inf\{|\xi(l)|>0 \mid l\in L\}>0$ as we noted in the definition of a \CAT\ group. In particular, conditions (i) and (ii) of the Main Theorem are satisfied for the action $\xi\colon L\to {\rm Isom}(X)$.

To see that this action satisfies conditions (iii) and (iv) of the Main Theorem, we apply Lemma \ref{intersection property} to the subgroup $\varphi(L)\subseteq G$. If all torsion subgroups of $G$ are finite, so are the torsion subgroups of $\varphi(L)$. Lemma \ref{intersection property} shows that the properties (iii) and (iv) of the Main Theorem are satisfied.

Due to Lemma \ref{max el}, we obtain the existence of a maximal element in any subfamily of $\{{\rm Fix}(\xi(H))\mid H\subseteq L \text{ closed subgroup}\}$. 

So we can apply the Main Theorem (3) to obtain that either $\xi$ is continuous or that there exists a closed subgroup $H'\subseteq L$ such that the non-empty proper fixed point set ${\rm Fix}(\xi(H'))$ is invariant under $\xi$.

If $\xi$ is continuous, then $\psi^{-1} \circ \xi =\varphi$ is continuous since $\psi $ is a homeomorphism.

Suppose now that ${\rm Fix}(\xi(L^\circ))\neq X$ (it is non-empty by Main Theorem (1)). That means $\varphi(L^\circ)\neq \{e_G\}$. For any $l\in L$ we additionally have the equality $\varphi(l)\varphi(L^\circ)\varphi(l)^{-1}=\varphi(lL^\circ l^{-1})=\varphi(L^\circ)$, so $\varphi(L)\subseteq {\rm Nor}_G(\varphi(L^\circ)) $. Since $\xi(L^\circ)$ has a global fixed point, $\varphi(L^\circ)$ has to be finite since the action is proper. 
	
	So we can assume ${\rm Fix}(\xi(L^\circ))=X$. This means the action factorizes through $\pi\colon L\to L/L^\circ$. 	
By assumption, the family 
$$\left\{ {\rm Fix} (\xi(\pi^{-1}(K)))\mid K\subseteq L/L^\circ \text{ compact open subgroup}\right\}$$ 
has maximal elements. Let ${\rm Fix}(\xi(\pi^{-1}(K_0)))$ be a maximal element such that the order of $\xi(\pi^{-1}(K_0))$ is minimal, note that the order is always finite, since ${\rm Fix}\left(\xi(\pi^{-1}(K_0))\right)$ is non-empty and the action is proper.  If the order of the subgroup $\xi(\pi^{-1}(K_0))$ is $1$, then $\xi$ and thus $\varphi$ is continuous. If the order of the subgroup $\xi(\pi^{-1}(K_0))$ is greater than $1$, then $\varphi(\pi^{-1}(K_0))$ is a non-trivial finite subgroup, so it suffices to show that $\varphi(L)\subseteq {\rm Nor}_G(\varphi(\pi^{-1}(K_0)))$. For $l\in L$ we have:
\begin{align*} \varphi(l)\varphi(\pi^{-1}(K_0)) \varphi(l)^{-1}&=\varphi\left(l\pi^{-1}(K_0)l^{-1}\right)\\
&=\varphi\left(\pi^{-1}\left(\pi(l)K_0\pi(l^{-1})\right)\right)\quad\text{since }\varphi(L^\circ)=\{e_G\}\\
&\overset{\star}{=}\varphi\left(\pi^{-1}\left(\pi(l)K_0\pi(l^{-1})\cap K_0\right)\right)\\
&\overset{\star\star}{=}\varphi(\pi^{-1}(K_0)).
\end{align*}
To see the equalities in $\star$ and $\star\star$ hold, we argue similarly to the Main Theorem: $\pi(l)K_0\pi(l)^{-1}$, $K_0$ and $\pi(l)K_0\pi(l)^{-1}\cap K_0$ are all compact open subgroups of $L/L^\circ$. The maximality of ${\rm Fix}\left(\xi(\pi^{-1}(K_0))\right)$ implies the equalities
$${\rm Fix}\left(\xi(\pi^{-1}(K_0))\right)={\rm Fix}\left(\xi(\pi^{-1}(\pi(l)K_0\pi(l)^{-1}\cap K_0)\right)={\rm Fix}\left(\xi(\pi^{-1}\left(\pi(l)K_0\pi(l)^{-1}\right)\right).$$
The minimality of the order of $\xi(\pi^{-1}(K_0))$ now converts the equalities of the fixed point sets into equalities of the groups.

For the in particular statement simply note that the groups in question are torsion free \CAT\ groups (see \cite[Prop. 1.1, Theorem]{AlibegovicBestvina} for limit groups and \cite[Theorem 3.1.1.]{CharneyDavis}, \cite{Charney} for RAAGs), so no non-trivial finite subgroups exist, hence the homomorphism has to be continuous.
\end{proof}
\begin{remark} Whether the condition on torsion subgroups of \CAT\ groups is an actual restriction is an open question. See for example \cite[Q 2.11]{BestvinaQ}, \cite[Q. 8.2]{BrdisonQ}, \cite[§ IV.5]{CapraceQ}, \cite[Prob 24]{KapovichQ}, \cite{OsajdaQ}, \cite{SwensonQ}.
\end{remark}

\end{document}